\documentclass[dvipdfmx,a4paper,12pt]{article}
\usepackage{geometry}
\usepackage[T1]{fontenc}
\usepackage{lmodern}
\usepackage{amsmath}
\usepackage{amssymb}
\usepackage{amsthm}
\usepackage{mathrsfs}
\usepackage{braket}
\usepackage{graphicx}
\usepackage{enumerate}
\usepackage[abbrev, shortalphabetic]{amsrefs}
\newcommand{\R}{\mathbb{R}}

\newcommand{\K}{\mathcal{K}}

\newcommand{\Span}{\operatorname{span}}

\newcommand{\conv}{\operatorname{conv}}
\newcommand{\pos}{\operatorname{pos}}
\newcommand{\sgn}{\operatorname{sgn}}

\newcommand{\interior}{\operatorname{int}}

\usepackage{bm}

\newcommand{\va}{\bm{a}}

\newcommand{\ve}{\bm{e}}
\newcommand{\vn}{\bm{n}}

\newcommand{\vr}{\bm{r}}
\newcommand{\vs}{\bm{s}}
\newcommand{\vv}{\bm{v}}
\newcommand{\vx}{\bm{x}}

\newcommand{\checkK}{\check{\mathcal{K}}}

\theoremstyle{plain}
\newtheorem{theorem}{Theorem}[section]
\newtheorem{lemma}[theorem]{Lemma}
\newtheorem{proposition}[theorem]{Proposition}
\newtheorem{corollary}[theorem]{Corollary}

\theoremstyle{definition}
\newtheorem{remark}[theorem]{Remark}

\newtheorem*{claim}{Claim}
\newtheorem{conjecture}{Conjecture}

\usepackage{xcolor}

\usepackage{multirow}
\usepackage{wasysym}

\begin{document}

\title{Minimal volume product of convex bodies with certain discrete symmetries and its applications}

\author{Hiroshi Iriyeh
\thanks{Graduate School of Science and Engineering, 
Ibaraki University, 
e-mail: hiroshi.irie.math@vc.ibaraki. ac.jp}
\and 
Masataka Shibata
\thanks{Department of Mathematics, Meijo University, 
e-mail: mshibata@meijo-u.ac.jp}
}

\maketitle

\begin{abstract}
We give the sharp lower bound of the volume product of $n$-dimensional convex bodies which are invariant under a discrete subgroup $SO(K)=\{ g \in SO(n); g(K)=K \}$, where $K$ is an $n$-cube or $n$-simplex.
This provides new partial results of Mahler's conjecture and its non-symmetric version.
In addition, we give partial answers for Viterbo's isoperimetric type conjecture in symplectic geometry from the view point of Mahler's conjecture.
\end{abstract}

\section{Introduction and main results}

A compact convex set in $\R^n$ with nonempty interior is called a {\it convex body}.
A convex body $K \subset \R^n$ is said to be {\it centrally symmetric} if it satisfies that $K=-K$.
Denote by $\mathcal{K}^n$ the set of all convex bodies in $\R^n$ equipped with the Hausdorff metric and by $\mathcal{K}^n_0$ the set of all $K \in \mathcal{K}^n$ which are centrally symmetric.
The interior of $K \in \mathcal{K}^n$ is denoted by $\mathrm{int}(K)$.
The {\it polar body} of $K$ with respect to $z \in \mathrm{int}(K)$ is defined by
\begin{equation*}
K^z :=
\left\{
y \in \R^n;
(y-z) \cdot (x-z) \leq 1 \text{ for any } x \in K
\right\},
\end{equation*}
where $\cdot$ denotes the standard inner product on $\R^n$.
Then the {\it volume product} of $K$ is defined by
\begin{equation}
\label{eq:a}
 \mathcal{P}(K) := \min_{z \in \mathrm{int}(K)}|K| \, |K^z|,
\end{equation}
where $|K|$ denotes the $n$-dimensional Lebesgue measure (volume) of $K$ in $R^n$.
Note that $\mathcal{P}(K)$ is invariant under non-singular affine transformations of $\R^n$.
It is known that for each $K \in \mathcal{K}^n$ the minimum of \eqref{eq:a} is attained at the unique point on $\mathrm{int}(K)$, which is called the {\it Santal\'o point} of $K$ (see e.g., \cite{MP}).
For a centrally symmetric convex body $K \in \mathcal{K}^n_0$, the Santal\'o point of $K$ is the origin $o$.
In the following, the polar of $K$ with respect to $o$ is denoted by $K^\circ$.

The following conjecture is a longstanding open problem concerning the volume product.

\begin{conjecture}[Mahler \cite{Ma1}]
\label{conj:1}
For any $K \in \mathcal{K}^n_0$, it satisfies
\begin{equation}
\label{eq:b}
 \mathcal{P}(K)=|K|\, |K^\circ| \geq \frac{4^n}{n!}.
\end{equation}
\end{conjecture}

Conjecture \ref{conj:1} was confirmed for $n=2$ (\cite{Ma2}) and $n=3$ (\cite{IS}).
The case that $n\geq4$ is still open, but there are many partial results (see, e.g., references in \cite{IS}).
An $n$-dimensional cube or, more generally, Hanner polytopes satisfy the equality in \eqref{eq:b} and these polytopes are predicted as the minimizers of $\mathcal{P}$.

There is another well-known conjecture for non-symmetric bodies as follows.

\begin{conjecture}
\label{conj:2}
Any $K \in \mathcal{K}^n$
satisfies that
\begin{equation*}
\mathcal{P}(K) \geq \frac{(n+1)^{n+1}}{(n!)^2}
\end{equation*}
with equality if and only if $K$ is a simplex.
\end{conjecture}

Conjecture \ref{conj:2} was proved by Mahler \cite{Ma2} for $n=2$ (see also \cite{Me}), but remains open for $n \geq 3$.

In this article, we first give partial answers of the above two conjectures, which relax the conditions of the results obtained in \cite{BF}.
Before we state our main results, we need some more notation.
Barthe and Fradelizi \cite{BF} obtained the sharp lower bound of $\mathcal{P}(K)$ when $K$ has many reflection symmetries, more precisely, when $K$ is invariant under the action of a Coxeter group.
The cross-polytope $\Diamond^n$, which is the polar of the $n$-cube $\square^n$, and $n$-simplex $\triangle^n$ with $o$ as the centroids are the candidates for the minimizers of the functional $\mathcal{P}$ in Conjectures \ref{conj:1} and \ref{conj:2}, respectively.
They have symmetries given by certain discrete subgroups of the orthogonal group $O(n)$.
Let $n \geq 2$ and $G$ be a (discrete) subgroup of $O(n)$.
We denote by  $\mathcal{K}^n(G)$ the set of $G$-invariant convex bodies in $\R^n$.
For a convex body $K \in \mathcal{K}^n$, we consider the subgroup of $SO(n)$ or $O(n)$ defined by
\begin{equation*}
 SO(K):=\left\{ g \in SO(n); gK=K \right\}, \ \ 
 O(K):=\left\{ g \in O(n); gK=K \right\}.
\end{equation*}
For example, as groups $O(\triangle^n)$ and $SO(\triangle^n)$ are isomorphic to the symmetric group and the alternating group of degree $n+1$, respectively.

Now we are in a position to state our results.

\begin{theorem}
\label{thm:1}
For any $K \in \mathcal{K}^n(SO(\Diamond^n))$, it holds that
\begin{equation*}
 |K|\, |K^\circ| \geq |\Diamond^n|\,|(\Diamond^n)^\circ| = \frac{4^n}{n!}.
\end{equation*}
\end{theorem}

\begin{theorem}
\label{thm:2}
For any $K \in \mathcal{K}^n(SO(\triangle^n))$, it holds that
\begin{equation*}
 |K|\, |K^\circ| \geq |\triangle^n|\,|(\triangle^n)^\circ| = \frac{(n+1)^{n+1}}{(n!)^2}.
\end{equation*}
\end{theorem}

\begin{remark}
(i) If $K \in \mathcal{K}^n(SO(\triangle^n))$ or $K \in \mathcal{K}^n(SO(\Diamond^n))$, then the centroid of $K$ is the origin $o$.
(ii) In \cite{BF}*{Theorem 1 (i)}, the above two results obtained under the assumptions that $K \in \mathcal{K}^n(O(\Diamond^n))$ and $K \in \mathcal{K}^n(O(\triangle^n))$, respectively.
(iii) When $n$ is odd, $-\operatorname{Id} \notin SO(\Diamond^n)$.
Hence, the assumption $K \in \mathcal{K}^n(SO(\Diamond^n))$ does not necessarily imply that $K$ is centrally symmetric.
(iv) For $n=3$, Theorems \ref{thm:1} and \ref{thm:2} were obtained in \cite{IS2}*{Theorem 1.5 (ii),\,(i)}.
\end{remark}

As we explain below, Mahler's conjecture (Conjecture \ref{conj:1}) is closely related to a conjecture in symplectic geometry, which concerns a symplectic invariant called symplectic capacities.
In fact, Theorem \ref{thm:1} has an application to the problem.

Let $(M,\omega)$ be a symplectic manifold, i.e., $M$ is a smooth manifold with a closed non-degenerate two-form $\omega$.
A symplectic capacity $c$ is a symplectic invariant which assigns a non-negative real number to each of symplectic manifolds of dimension $2n$, which satisfies the following three conditions: (i) If there exists a symplectic embedding from $(M,\omega)$ to $(N,\tau)$, then $c(M,\omega) \leq c(N,\tau)$.
(ii) $c(M,\alpha\omega)=|\alpha|\,c(M,\omega)$ for any $\alpha \in \R \setminus \{ 0 \}$.
(iii) $c(B^{2n},\omega_0)=c(Z^{2n},\omega_0)=\pi$, where $\omega_0$ is the standard symplectic form on ${\mathbb R}^{2n}$, $B^{2n}$ denotes the $2n$-dimensional open unit ball and $Z^{2n}:=B^2 \times \R^{2n-2}$ is the symplectic cylinder.
For details, see e.g., \cite{AKO}.
The following is an isoperimetric-type conjecture for symplectic capacities of convex domains in $\R^{2n}$ with the standard symplectic structure $\omega_0$.
From now on, we put $c(\Sigma):=c(\Sigma,\omega_0)$ for a domain $\Sigma \subset (\R^{2n},\omega_0)$.

\begin{conjecture}[Viterbo \cite{V}]
\label{conj:3}
For any symplectic capacity $c$ and any convex domain
$\Sigma \in \mathcal{K}^{2n}$,
\begin{equation}
\label{eq:c}
 \frac{c(\Sigma)}{c(B^{2n})} \leq \left(\frac{\mathrm{vol}(\Sigma)}{\mathrm{vol}(B^{2n})}\right)^{1/n}
\end{equation}
holds, where $\mathrm{vol}(\Sigma)$ denotes the symplectic volume of $\Sigma$.
\end{conjecture}
Since $c(B^{2n})=\pi$ and $\mathrm{vol}(B^{2n})=\pi^n/n!$, inequality \eqref{eq:c} is equivalent to $c(\Sigma)^n \leq n!\,\mathrm{vol}(\Sigma)$.
Note that this conjecture is unsolved even in the case of $n=2$.
However, it is easily seen that the inequality \eqref{eq:c} holds for the {\it Gromov width} $c=c_G$, which is defined by
\begin{equation*}
c_G(\Sigma) := \sup\{ \pi r^2 \in (0, \infty] ; \mbox{there is a symplectic embedding from} \ (B^{2n}(r),\omega_0) \ \mathrm{to} \ \Sigma \},
\end{equation*}
where $B^{2n}(r)$ denotes the $2n$-dimensional open ball with radius $r$.
The Hofer-Zehnder capacity $c_{\rm HZ}$ is one of the important symplectic
capacities, which is related to Hamiltonian dynamics on symplectic manifolds.
Artstein-Avidan, Karasev and Ostrover calculated this capacity of the following special type of convex domains.

\begin{theorem}[\cite{AKO}, Theorem 1.7]
\label{thm:a}
For any centrally symmetric convex body $K \in \mathcal{K}_0^n$,
we have
\begin{equation*}
c_{\rm HZ}(K \times K^\circ) = 4.
\end{equation*}
\end{theorem}

As a consequence, they gave a remarkable observation that {\it Viterbo's conjecture implies Mahler's conjecture} (see \cite{AKO}*{Section 1}).
Conversely, if Mahler's conjecture is true, by Theorem \ref{thm:a}, Viterbo's conjecture is true for $\Sigma=K \times K^\circ$ with respect to $c=c_{\rm HZ}$.
In general, a convex domain $K \times L$ in $(\R^{2n}=\R_x^n \times \R_y^n,\omega_0=dx \wedge dy)$, where $K, L \subset \R^n$, is called a {\it Lagrangian product}.

The second purpose of this article is to show that if Mahler's conjecture is true for $K \in \K_0^n$, then Viterbo's conjecture is also true with respect to {\it any} symplectic capacity $c$ for the class of convex symmetric Lagrangian products.

\begin{proposition}
\label{prop:v}
Let $L \in \K_0^n$.
If Mahler's conjecture is true for $K \in \K_0^n$, then Conjecture \ref{conj:3} is true for the Lagrangian product $K \times L \subset (\R^{2n},\omega_0)$ with respect to any symplectic capacity $c$.
\end{proposition}

For instance, Theorem \ref{thm:1} yields

\begin{corollary}
\label{cor:4}
Let $K \in \mathcal{K}^n(SO(\Diamond^n))$ and $L \subset {\mathbb R}^n$ be a centrally symmetric convex body.
Assume that $n$ is even.
Then for the Lagrangian product $\Sigma:=K \times L \subset ({\mathbb R}^{2n},\omega_0)$ the inequality $c(\Sigma)^n \leq n!\,\mathrm{vol}(\Sigma)$ holds for any symplectic capacity $c$.
\end{corollary}

\begin{remark}
When $n$ is odd, we do not know whether the same result holds, since $-\operatorname{Id} \notin SO(\Diamond^n)$.
At least, in this case the statement of Corollary \ref{cor:4} holds if we replace the assumption that $K \in \mathcal{K}^n(SO(\Diamond^n))$ by $K \in \mathcal{K}^n(O(\Diamond^n))$, because Mahler's conjecture is true for $K \in \mathcal{K}^n(O(\Diamond^n))$ (see \cite{BF}*{Theorem 1 (i)}).
\end{remark}

This paper is organized as follows.
In Section 2, we give a detailed exposition of the higher dimensional version of the ``signed volume estimate'' originally introduced in \cite{IS}, \cite{IS2} for the three dimensional case, which generalizes the methods used in \cite{M} and \cite{BF}.
In Sections 3 and 4, we prove Theorem \ref{thm:1} and Theorem \ref{thm:2}, respectively, by means of the estimates prepared in Section 2.
In Section 5, we prove Proposition \ref{prop:v} and provide several partial results of Conjecture \ref{conj:3} deduced from Proposition \ref{prop:v}.

\section{Preliminaries}

\subsection{Notations}

In the most part of the arguments in this article, we work with the following class of convex bodies in $\R^n$:
\begin{equation*}
\checkK^n:=\left\{
K \in \mathcal{K}^n; \text{$K$ is strongly convex and  $\partial K$ is of class $C^\infty, o \in \mathrm{int} K$}\right\}.
\end{equation*}
Let us consider a convex body $K \in \checkK^n$.
The {\it radial function} of $K$ is defined by $\rho_K(\vx) := \max\{ \lambda \geq 0; \lambda \vx \in K \}$ for $\vx \in \R^n \setminus \{ o \}$.
Then the gauge function $\mu_K$ of $K$ satisfies that $\mu_K(\vx)=1/\rho_K(\vx)$.
In what follows, for every $k$ vectors $\va_1, \dots, \va_k \in \R^n$ we denote by $\conv \{\va_1, \dots, \va_k\}$ the convex hull of $\va_1, \dots, \va_k$ and by $\pos \{\va_1, \dots, \va_k\}$ the positive hull of $\va_1, \dots, \va_k$.
Moreover, in case $\va_1, \dots, \va_k$ are linearly independent, we define the following subset of $\partial K$:
\begin{equation*}
\mathcal{C}(\va_1, \dots, \va_k)=
\mathcal{C}_K(\va_1, \dots, \va_k):=
\left\{
\rho_K(\vx) \vx ; 
\vx \in 
\conv \{\va_1, \dots, \va_k\}
\right\}.
\end{equation*}
The order of the vectors $\va_1, \dots, \va_k$ induces a natural orientation of $\mathcal{C}(\va_1, \dots, \va_k)$.
We always fix the orientation on it.
We denote by $-\mathcal{C}(\va_1, \dots, \va_k)$ the set $\mathcal{C}(\va_1, \dots, \va_k)$ with the opposite orientation.
Then for any $g \in O(n)$, we have
\begin{equation*}
 g \mathcal{C}_K(\va_1, \dots, \va_k)
=
(\det g) \mathcal{C}_{gK}(g \va_1, \dots, g\va_k).
\end{equation*}
If $K$ is invariant under the action of $g$, i.e.\ it satisfies $gK=K$, then
\begin{equation*}
 g \mathcal{C}_K(\va_1, \dots, \va_k)
=
(\det g) \mathcal{C}_K(g \va_1, \dots, g\va_k)
\end{equation*}
holds.
Moreover, since $\partial K$ is of class $C^\infty$, then $\mathcal{C}_K(\va_1, \dots, \va_k)$ is a $(k-1)$-dimensional smooth oriented submanifold of $\partial K$ with a boundary which is piecewise smooth.

For a convex body $K \in \checkK^n$, its polar $K^\circ$ is also strongly convex with $\partial K^\circ$ of class $C^\infty$.
Denote by $\nabla \mu_K : \R^n \to \R^n$ the gradient of the gauge function $\mu_K$.
Then we can consider a map $\Lambda=\Lambda_K: \partial K \to \partial K^\circ$ defined by
\begin{equation*}
\Lambda_K(\vx)=\nabla \mu_K(\vx) \quad (\vx \in \partial K),
\end{equation*}
which is a $C^\infty$-diffeomorphism satisfying that $\vx \cdot \Lambda(\vx)=1$ (see \cite{Sc}*{Section 1.7.2}).
Let $G$ be a subgroup of $O(n)$.
We denote by $\checkK^n(G)$ the set of all convex bodies $K \in \checkK^n$ satisfying that $K=GK$.
\begin{lemma}
\label{lem:1}
Let $K \in \checkK^n(G)$.
For any $\vx \in \partial K$ and any $g \in G$,
\begin{equation}
\label{eq:16}
g \Lambda_K(\vx)
= \Lambda_K(g \vx)
\end{equation}
holds.
\end{lemma}

\begin{proof}
The proof of the equality \eqref{eq:16} is the same as the three dimensional case \cite{IS2}*{Lemma 3.2}.
So we omit it.
\end{proof}

Let $K \in \checkK^n$ $(n \geq 3)$ and $C$ be an oriented $(n-2)$-dimensional submanifold in $\R^n$ satisfying that $C \subset \partial K$.
Suppose that $C$ is piecewise $C^\infty$ and equipped with a parametrization $\vr(t_1,\dots,t_{n-2})$, $(t_1,\dots, t_{n-2}) \in D$ from a simply connected domain $D \subset \R^{n-2}$ to $\partial K \subset \R^n$, where we choose $\vr$ such that $\vr : D \to C$ preserves the orientations.
In this setting, we define the vector $\overline{C} \in \R^n$ by the equality
\begin{equation}
\label{eq:13}
\overline{C} \cdot \vx = \frac{1}{n-1} \int_D
\left[
\vx, \vr, \frac{\partial \vr}{\partial t_1}, \dots,
\frac{\partial \vr}{\partial t_{n-2}}
\right]
 \,dt_1 \cdots dt_{n-2} \ \text{ for } \ \vx \in \R^n,
\end{equation}
where
\begin{equation*}
\left[
\vv_1,\vv_2,\dots,\vv_n
\right]:=\det
\begin{pmatrix}
\vv_1 & \vv_2 & \cdots & \vv_n
\end{pmatrix}
\ \text{ for } \ \vv_i \in \R^n \, (i=1,\dots,n).
\end{equation*}
It is easy to check that the right hand side of \eqref{eq:13} is independent of the choice of the parametrization $\vr(t_1,\dots,t_{n-2})$, so that the vector $\overline{C}$ is well-defined.
We denote by $-C$ the submanifold $C$ with the opposite orientation.
By the definition of $\overline{C}$,
\begin{equation*}
 \overline{-C}=-\overline{C}
\end{equation*}
holds.
Moreover, for any $R \in O(n)$ and $\vx \in \R^n$ we obtain
\begin{equation*}
\overline{R C} \cdot (R \vx) 
=
(\det R) 
\overline{C} \cdot \vx 
=
(\det R) 
(R \overline{C}) \cdot (R \vx),
\end{equation*}
hence
\begin{equation*}
 \overline{R C} = (\det R) R \overline{C}
\end{equation*}
holds.

\begin{lemma}
\label{lem:2}
Let $K \in \checkK^n$.
Assume that $\va_1, \dots, \va_{n-1} \in \R^n$ are linearly independent.
Let us consider $\vs(t_1,\dots,t_{n-2}) := \va_1 + t_1 (\va_2 - \va_1) + \dots + t_{n-2}(\va_{n-1} - \va_1)$, $(t_1,\dots,t_{n-2}) \in D$, where
\begin{equation*}
D:=\left\{(t_1,\dots,t_{n-2}) \in \R^{n-2}; t_1,\dots, t_{n-2} \geq 0, t_1 + \dots + t_{n-2} \leq 1 \right\}.
\end{equation*}
Then we have
\begin{enumerate}[\upshape (i)]
\item $\displaystyle \overline{\mathcal{C}(\va_1, \dots, \va_{n-1})} \cdot \vx
= \int_D \rho_K^{n-1}(\vs(t_1,\dots,t_{n-2}))
\left[ \vx, \va_1,\dots,\va_{n-1} \right] \,dt_1 \cdots dt_{n-2}$ \ for any $\vx \in \R$.
\item 
$\overline{\mathcal{C}(\va_{\sigma(1)}, \dots, \va_{\sigma(n-1)})} = (\sgn \sigma) \overline{\mathcal{C}(\va_1, \dots, \va_{n-1})}$ \ for $\sigma \in \mathfrak{S}_{n-1}$, where $\mathfrak{S}_{n-1}$ denotes the symmetric group of order $(n-1)$ and $\sgn \sigma$ denotes the signature of the permutation of $\sigma$.
\end{enumerate}
\end{lemma}

\begin{proof}
Putting $\vr(t_1,\dots,t_{n-2}) := \rho_K(\vs(t_1,\dots,t_{n-2})) \vs(t_1,\dots,t_{n-2})$, then $\vr(t_1,\dots,t_{n-2}) \in \partial K$ gives a parametrization of $\mathcal{C}(\va_1, \dots, \va_{n-1})$ and we obtain
\begin{equation*}
\begin{aligned}
& \left[
\vx, \vr, \frac{\partial \vr}{\partial t_1}, \dots
\frac{\partial \vr}{\partial t_{n-2}}
\right] \\
&=
\left[
\vx, \rho_K(\vs) \vs, 
\frac{\partial}{\partial t_1}(\rho_K(\vs)) \vs + \rho_K(\vs) (\va_2-\va_1), \dots,
\frac{\partial}{\partial t_{n-2}}(\rho_K(\vs)) \vs + \rho_K(\vs) (\va_{n-1}-\va_1)
\right] \\
&=
\left[
\vx, \rho_K(\vs) \vs, 
\rho_K(\vs) (\va_2-\va_1), \dots,
\rho_K(\vs) (\va_{n-1}-\va_1)
\right] \\
&=
\rho_K^{n-1}(\vs)
\left[
\vx, \va_1 + t_1 (\va_2-\va_1) + \dots + t_{n-2}(\va_{n-1}-\va_1), 
\va_2-\va_1, \dots,
\va_{n-1}-\va_1
\right] \\
&=
\rho_K^{n-1}(\vs)
\left[
\vx, \va_1,
\va_2-\va_1, \dots,
\va_{n-1}-\va_1
\right] \\
&=
\rho_K^{n-1}(\vs)
\left[
\vx, \va_1,
\va_2, \dots,
\va_{n-1}
\right].
\end{aligned}
\end{equation*}
By the definition of $\overline{\mathcal{C}(\va_1, \dots, \va_{n-1})}$, (i) holds.
Property (ii) follows from the equality
\begin{equation*}
 [\vx, \va_{\sigma(1)}, \dots, \va_{\sigma(n-1)}] =
(\sgn \sigma)
 [\vx, \va_1, \dots, \va_{n-1}].
\end{equation*}
\end{proof}

For a convex body $K \in \checkK^n$, the section $K \cap H$ of $K$ by a hyperplane $(o \in) H \subset \R^n$ is a convex body in $H (\simeq \R^{n-1})$ and belongs to the class $\checkK^{n-1}$.
Although every subset $C \subset \partial K \cap H$ is contained in the hyperplane $H$, its image $\Lambda(C) \subset \partial K^\circ$ is not necessarily contained in a hyperplane of $\R^n$.
However, the following lemma enables us to analyze the image $\Lambda(C)$ in the case where $C=\mathcal{C}(\va_1, \dots, \va_{n-1})$.

\begin{lemma}
\label{lem:7}
Let $K \in \checkK^n$ and $\va_1, \dots, \va_{n-1} \in \R^n$ be linearly independent.
We put $H:=\Span \{\va_1, \dots, \va_{n-1}\}$ and denote by $\pi_H$ the projection from $\R^n$ to $H$.
Then the following properties hold.
\begin{enumerate}[\upshape (i)]
\item The truncated convex cone $o*\mathcal{C}(\va_1, \dots, \va_{n-1}) = K \cap \pos \left\{\va_1, \dots, \va_{n-1}\right\}$ is a subset of the hyperplane $H$ in $\R^n$ and the vector $\overline{\mathcal{C}(\va_1, \dots, \va_{n-1})} \in \R^n$ is perpendicular to $\va_1, \dots, \va_{n-1}$.
\item On $H$, $o*\Lambda_{K \cap H}(\mathcal{C}(\va_1, \dots, \va_{n-1})) = o*\pi_H \Lambda_K(\mathcal{C}(\va_1, \dots, \va_{n-1}))$ holds, where $\Lambda_{K \cap H} := \nabla \mu_{K \cap H}|_{\partial(K \cap H)}$.
\item The equality
\begin{equation*}
\begin{aligned}
&\overline{\mathcal{C}(\va_1, \dots, \va_{n-1})} \cdot
\overline{\Lambda_K(\mathcal{C}(\va_1, \dots, \va_{n-1}))} \\
&=
|o*\mathcal{C}(\va_1, \dots, \va_{n-1})|_{n-1} \,
|o*\pi_H \Lambda_K(\mathcal{C}(\va_1, \dots, \va_{n-1}))|_{n-1}
\end{aligned}
\end{equation*}
holds, where $|\cdot|_{n-1}$ denotes the $(n-1)$-dimensional Lebesgue measure (volume).
\end{enumerate}
\end{lemma}

\begin{proof}
(i) It is a direct consequence of the definition of $\mathcal{C}(\va_1, \dots, \va_{n-1})$ and Lemma 2.2 (i).

(ii) It suffices to consider the case that $H=\{\vx=(x_1,\dots,x_{n-1},x_n) \in \R^n; x_n=0\}$.
By definition, $\Lambda_K = \nabla \mu_K|_{\partial K}$ is a $C^\infty$-diffeomorphism from $\partial K$ to $\partial K^\circ$, where $\nabla \mu_K = (\partial_{x_1} \mu_K, \dots, \partial_{x_n} \mu_K)$.
Setting $L:=K \cap H$, then $o \in \interior L \subset H$, $L \in \checkK^{n-1}$ and $\mu_L = \mu_K|_H$.
Since $\nabla \mu_L = (\partial_{x_1}(\mu_K|_H), \dots, \partial_{x_{n-1}}(\mu_K|_H)): H \to H$ and $\Lambda_L := \nabla \mu_L|_{\partial L}: \partial L \to \Lambda_L(\partial L)$, we have
\begin{equation*}
\pi_H \circ \Lambda_K|_{\partial L} = (\partial_{x_1} \mu_K, \dots, \partial_{x_{n-1}} \mu_K)|_{\partial L} = \Lambda_L
\end{equation*}
on $\partial L$.
Applying it to $\mathcal{C}(\va_1, \dots, \va_{n-1}) \subset \partial L$, we have
\begin{equation*}
o*\pi_H(\Lambda_K(\mathcal{C}(\va_1, \dots, \va_{n-1})))
= o*\Lambda_{K \cap H}(\mathcal{C}(\va_1, \dots, \va_{n-1})).
\end{equation*}

(iii) Since the boundary of $K$ is of class $C^\infty$, the relative interior of the submanifold $\mathcal{C}(\va_1, \dots, \va_{n-1})$ on $\partial K$ is of class $C^\infty$ and, by definition, we can represent the vector $\overline{\mathcal{C}(\va_1, \dots, \va_{n-1})} \in \R^n$ as
\begin{equation*}
\overline{\mathcal{C}(\va_1, \dots, \va_{n-1})} = |o*\mathcal{C}(\va_1, \dots, \va_{n-1})|_{n-1} \vn,
\end{equation*}
where $\vn$ is one of the unit normal vectors of the hyperplane $H$.

On the other hand, let $\vr(t_1,\dots,t_{n-2})$, $(t_1,\dots, t_{n-2}) \in D$ be a parametrization of the $C^\infty$-submanifold $\Lambda_K(\mathcal{C}(\va_1, \dots, \va_{n-1})) \subset \partial K^\circ$ with a piecewise $C^\infty$ boundary.
Then, by \eqref{eq:13}, we have
\begin{equation*}
\overline{\Lambda_K(\mathcal{C}(\va_1, \dots, \va_{n-1}))} \cdot \vn = \frac{1}{n-1} \int_D
\left[
\vn, \vr, \frac{\partial \vr}{\partial t_1}, \dots,
\frac{\partial \vr}{\partial t_{n-2}}
\right]
 \,dt_1 \cdots dt_{n-2}.
\end{equation*}
This quantity is nothing but the (singed) volume of the projection image of the submanifold $o*\Lambda(\mathcal{C}(\va_1, \dots, \va_{n-1}))$ to $H$, which is a convex set in $H$ by (ii).
Consequently,
\begin{equation*}
\begin{aligned}
&\overline{\mathcal{C}(\va_1, \dots, \va_{n-1})} \cdot
\overline{\Lambda_K(\mathcal{C}(\va_1, \dots, \va_{n-1}))} \\
&=
|o*\mathcal{C}(\va_1, \dots, \va_{n-1})|_{n-1} \vn \cdot
\overline{\Lambda_K(\mathcal{C}(\va_1, \dots, \va_{n-1}))} \\
&=
|o*\mathcal{C}(\va_1, \dots, \va_{n-1})|_{n-1} \,
|o*\pi_H \Lambda_K(\mathcal{C}(\va_1, \dots, \va_{n-1}))|_{n-1}
\end{aligned}
\end{equation*}
\end{proof}

\subsection{The signed volume estimate}

The following is a higher dimensional generalization of the formula \cite{IS}*{Proposition 3.2} (see also \cite{IS2}*{Lemma 3.8}).
Note that under a slightly different situation the essentially same formula was obtained in \cite{FHMRZ}*{Proposition 1}.
Here we shall give a formulation and its proof along the setting of the present article.

\begin{proposition}
\label{prop:11}
Let $K \in \checkK^n$ and $B$ be a subset of $\partial K$ which is homeomorphic to the $(n-1)$-dimensional closed ball.
Assume that the relative boundary $C$ of $B$ in $\partial K$ is a piecewise smooth $(n-2)$-dimensional manifold with the natural orientation induced from that of $B$ and each smooth part of $C$ possesses a parametrization from a simply connected domain.
Then
\begin{equation*}
 \frac{\overline{C} \cdot \vx}{n} \leq |o*B| \ \text{ for } \vx \in K
\end{equation*}
holds.
\end{proposition}

\begin{proof}
Let $(y_1,\cdots,y_n)$ be the standard coordinates of $\R^n$.
We put $E:=o*B$ and denote by $\iota$ the inclusion $\partial E \subset \R^n$.
Let $\vr$ be the position vector field of the piecewise $C^\infty$-hypersurface $\partial E \subset \R^n$ and $\vn$ be its outward unit normal vector field at regular points $\vr$.
By Stokes' theorem, we obtain
\begin{equation*}
\begin{aligned}
n |E|&=
n \int_E 1\, dy_1 \cdots dy_n \\
&=
\int_E \sum_{i=1}^n (-1)^{i-1}
d(y_i dy_1 \wedge \cdots \wedge dy_{i-1} \wedge dy_{i+1} \wedge \cdots \wedge dy_n) \\
&=
\int_{\partial E} \sum_{i=1}^n (-1)^{i-1}
\iota^* (y_i dy_1 \wedge \cdots \wedge dy_{i-1} \wedge dy_{i+1} \wedge \cdots \wedge dy_n) \\
&=
\int_{\partial E} \vr \cdot \vn\,\Omega,
\end{aligned}
\end{equation*}
where the last equality is due to \cite{Mat}*{p.\,285, Problem 1}, and $\Omega$ is the volume element of the hypersurface $\partial E$ with respect to the induced Riemannian metric from $\R^n$.

Next, we represent $\partial E$ as the union of the cone part $o*C$ and $B(=\partial E \cap \partial K)$.
Let $\vx=(x_1,\ldots,x_n)$ be an arbitrary vector in $K$.
Since $\vr \cdot \vn=0$ on $o*C$, and for every normal vector $\vn$ (at $\vr \in B$) the function $\vx \cdot \vn$ on $K$ maximize at $\vx=\vr$, we have
\begin{equation}
\label{eq:30}
\begin{aligned}
n |E|&=
\int_{(o*C) \cup B} \vr \cdot \vn\,\Omega
=
\int_B \vr \cdot \vn\,\Omega
\geq
\int_B \vx \cdot \vn\,\Omega \\
&=
\int_B \sum_{i=1}^n (-1)^{i-1}
x_i \iota^* (dy_1 \wedge \cdots \wedge dy_{i-1} \wedge dy_{i+1} \wedge \cdots \wedge dy_n).
\end{aligned}
\end{equation}
In the last equality we used the formula \cite{Mat}*{p.\,285, Problem 1} again.
In order to analyze the integrand, we consider the $(n-1)$-form
\begin{equation*}
\xi:=
\sum_{i=1}^n (-1)^{i-1}
x_i \iota^* (dy_1 \wedge \cdots \wedge dy_{i-1} \wedge dy_{i+1} \wedge \cdots \wedge dy_n).
\end{equation*}
For each $i=1,\ldots,n$, the $(n-2)$-form
\begin{equation*}
\begin{aligned}
\zeta_i:=&
\sum_{j<i} (-1)^{j-1}
y_j dy_1 \wedge \cdots \wedge dy_{j-1} \wedge dy_{j+1} \wedge \cdots \wedge dy_{i-1} \wedge dy_{i+1} \wedge \cdots \wedge dy_n \\
&+
\sum_{j>i} (-1)^{j-2}
y_j dy_1 \wedge \cdots \wedge dy_{i-1} \wedge dy_{i+1} \wedge \cdots \wedge dy_{j-1} \wedge dy_{j+1} \wedge \cdots \wedge dy_n
\end{aligned}
\end{equation*}
yields
$d\zeta_i=(n-1)dy_1 \wedge \cdots \wedge dy_{i-1} \wedge dy_{i+1} \wedge \cdots \wedge dy_n$, so that
\begin{equation*}
d\zeta=\xi, \quad \mbox{where} \quad \zeta:=\frac{1}{n-1}\sum_{i=1}^n (-1)^{i-1} x_i \iota^* \zeta_i
\end{equation*}
holds.
Therefore, it follows from the inequality \eqref{eq:30} that
\begin{equation}
\label{eq:31}
n |E| \geq \int_B \xi = \int_B d\zeta = \int_C \zeta.
\end{equation}

On the other hand, $C$ is an union of finitely many smooth part $C_i$ and the position vector field $\vr=(y_1,\ldots,y_n)$ of $C_i$ is represented as $y_i=y_i(t_1,\dots, t_{n-2})$, $(t_1,\dots, t_{n-2}) \in D_i$, where each $D_i$ is a contractible domain in $\R^{n-2}$.
In order to relate the last term of \eqref{eq:31} to the definition of $\overline{C}(=\sum_i \overline{C}_i)$, we consider the $(n-2)$-form
\begin{equation*}
\left[
\vx, \vr, \frac{\partial \vr}{\partial t_1}, \dots,
\frac{\partial \vr}{\partial t_{n-2}}
\right]
 \,dt_1 \wedge \cdots \wedge dt_{n-2}.
\end{equation*}
By the cofactor expansion with respect to the first column, we have
\begin{equation*}
\left[
\vx, \vr, \frac{\partial \vr}{\partial t_1}, \dots,
\frac{\partial \vr}{\partial t_{n-2}}
\right]
=
\sum_{i=1}^n (-1)^{i-1} x_i
\begin{vmatrix}
y_1 & \frac{\partial y_1}{\partial t_1} & \cdots & \frac{\partial y_1}{\partial t_{n-2}} \\
\cdots & \cdots & \cdots & \cdots \\
y_{i-1} & \frac{\partial y_{i-1}}{\partial t_1} & \cdots & \frac{\partial y_{i-1}}{\partial t_{n-2}} \\
y_{i+1} & \frac{\partial y_{i+1}}{\partial t_1} & \cdots & \frac{\partial y_{i+1}}{\partial t_{n-2}} \\
\cdots & \cdots & \cdots & \cdots \\
y_n & \frac{\partial y_n}{\partial t_1} & \cdots & \frac{\partial y_n}{\partial t_{n-2}}
\end{vmatrix}.
\end{equation*}
Once more we apply the cofactor expansion to each determinant of the right-hand side.
Then we have
\begin{equation*}
\Bigg{\{}
\sum_{j<i} (-1)^{j-1} y_j
\begin{vmatrix}
\frac{\partial y_1}{\partial t_1} & \cdots & \frac{\partial y_1}{\partial t_{n-2}} \\
\cdots & \cdots & \cdots \\
\frac{\partial y_{j-1}}{\partial t_1} & \cdots & \frac{\partial y_{j-1}}{\partial t_{n-2}} \\
\frac{\partial y_{j+1}}{\partial t_1} & \cdots & \frac{\partial y_{j+1}}{\partial t_{n-2}} \\
\cdots & \cdots & \cdots \\
\frac{\partial y_{i-1}}{\partial t_1} & \cdots & \frac{\partial y_{i-1}}{\partial t_{n-2}} \\
\frac{\partial y_{i+1}}{\partial t_1} & \cdots & \frac{\partial y_{i+1}}{\partial t_{n-2}} \\
\cdots & \cdots & \cdots \\
\frac{\partial y_n}{\partial t_1} & \cdots & \frac{\partial y_n}{\partial t_{n-2}}
\end{vmatrix}
+
\sum_{j>i} (-1)^{j-2} y_j
\begin{vmatrix}
\frac{\partial y_1}{\partial t_1} & \cdots & \frac{\partial y_1}{\partial t_{n-2}} \\
\cdots & \cdots & \cdots \\
\frac{\partial y_{i-1}}{\partial t_1} & \cdots & \frac{\partial y_{i-1}}{\partial t_{n-2}} \\
\frac{\partial y_{i+1}}{\partial t_1} & \cdots & \frac{\partial y_{i+1}}{\partial t_{n-2}} \\
\cdots & \cdots & \cdots \\
\frac{\partial y_{j-1}}{\partial t_1} & \cdots & \frac{\partial y_{j-1}}{\partial t_{n-2}} \\
\frac{\partial y_{j+1}}{\partial t_1} & \cdots & \frac{\partial y_{j+1}}{\partial t_{n-2}} \\
\cdots & \cdots & \cdots \\
\frac{\partial y_n}{\partial t_1} & \cdots & \frac{\partial y_n}{\partial t_{n-2}}
\end{vmatrix}
\Bigg{\}}
dt_1 \wedge \cdots \wedge dt_{n-2},
\end{equation*}
which is equal to $\iota^* \zeta_i$.
Hence,
\begin{equation}
\label{eq:32}
\left[
\vx, \vr, \frac{\partial \vr}{\partial t_1}, \dots,
\frac{\partial \vr}{\partial t_{n-2}}
\right]
 \,dt_1 \wedge \cdots \wedge dt_{n-2}
=
\sum_{i=1}^n (-1)^{i-1} x_i \iota^* \zeta_i
=
(n-1)\zeta.
\end{equation}
Consequently, we obtain $n |E| \geq \int_C \zeta=\overline{C} \cdot \vx$ from \eqref{eq:31} and \eqref{eq:32}.
\end{proof}

\subsection{Barth-Fradelizi's estimate}

The contents of this subsection is a counterpart of the estimate in \cite{BF}*{Lemma 11} in order to adapt to the setting of this article.
Essentially, nothing is new except the case where $k=n-1$ (see Lemma \ref{lem:8} and Remark 2.6 below).

We first take $n$ vectors $\va_1, \dots,\va_n \in \R^n$ such that
\begin{equation*}
\va_i \cdot \va_j =
\begin{cases}
 1 & \text{ if } i=j, \\
 \alpha & \text{ if } i \not= j,
\end{cases}
\end{equation*}
where $|\alpha|<1$ and $\alpha \neq -1/(n-1)$.
Then $\va_1, \dots, \va_n$ are linearly independent.
For example, if $\alpha=-1/n$, then $o*\conv \{\va_1, \dots, \va_n\}$ represents a fundamental domain of a simplex $\triangle^n$, so that $\triangle^n=\cup_{g \in SO(\triangle^n)}\,g (o*\conv \{\va_1, \dots, \va_n\})$.
And if $\alpha=0$, then $o*\conv \{\va_1, \dots, \va_n\}$ represents that of a cross polytope $\Diamond^n$.
Now we let $K \in \checkK^n$, then the $C^\infty$-diffeomorphism $\Lambda=\Lambda_K: \partial K \to \partial K^\circ$ is well-defined.
For $1 \leq k \leq n$ and $1 \leq i_1 < \dots < i_k \leq n$, the subset $\mathcal{C}(\va_{i_1}, \dots, \va_{i_k})$ of $\Span\{\va_{i_1}, \dots, \va_{i_k}\} \cap \partial K$ is a $(k-1)$-dimensional manifold with a piecewise $C^\infty$ boundary.
Note that $\mathcal{C}(\va_{i_1})=\rho_K(\va_{i_1})\va_{i_1} \in \partial K$.
In this setting, we have the following:

\begin{lemma}
\label{lem:8}
For $2 \leq k \leq n$, we put $H:=\Span\{\va_1, \dots, \va_k\} \simeq \R^k$.
Denote by $\pi_H$ the projection from $\R^n$ to $H$.
Assume that
\begin{enumerate}[\upshape (i)]
 \item $|o*\mathcal{C}(\va_1, \dots, \va_{i-1}, \va_{i+1}, \dots, \va_k)|_{k-1}$ is independent of $i$,
 \item $\Lambda_K(\mathcal{C}(\va_1, \dots, \va_{i-1}, \va_{i+1}, \dots, \va_k)) \subset \Span \{\va_1, \dots, \va_{i-1}, \va_{i+1}, \dots, \va_k\}$ for each $i$,
 \item $|o*\Lambda_K(\mathcal{C}(\va_1, \dots, \va_{i-1}, \va_{i+1}, \dots, \va_k))|_{k-1}$ is independent of $i$,
\end{enumerate}
where $1 \leq i \leq k$.
Then we have
\begin{equation}
\label{eq:BF}
\begin{aligned}
&|o*\mathcal{C}(\va_1, \dots, \va_k)|_{k}
|o*\pi_H \Lambda_K(\mathcal{C}(\va_1, \dots, \va_k))|_{k} \\
&\geq
\frac{1-\alpha}{k(1+(k-2)\alpha)}
|o*\mathcal{C}(\va_1, \dots, \va_{k-1})|_{k-1}
|o*\Lambda_K(\mathcal{C}(\va_1, \dots, \va_{k-1}))|_{k-1},
\end{aligned}
\end{equation}
where $|\cdot|_k$ denotes the $k$-dimensional Lebesgue measure on $H \simeq \R^k$ and $|\cdot|_{k-1}$ means the $(k-1)$-dimensional Lebesgue measure on every codimension one subspace of $H$.
\end{lemma}

\begin{remark}
We will use the inequality \eqref{eq:BF} to prove Theorems \ref{thm:1} and \ref{thm:2}.
However, \eqref{eq:BF} is useless for $k=n$ and this is why we need the ``signed volume estimate''.
\end{remark}

\begin{proof}
We put $L:=K \cap H$, then $L \in \checkK^k$ and the polar $L^\circ$ of $L$ (in $H$) is given by $L^\circ = \pi_H K^\circ$.
Indeed, for $k=n-1$ this fact is well-known (see e.g., \cite{GMR}*{p.\,274}).
When $H$ has higher codimension in $\R^n$, we can prove it by using the codimension one case inductively.
Note that $\mathcal{C}(\va_1, \dots, \va_k)$ is homeomorphic to $(k-1)$-dimensional closed ball in $\partial L$, of which boundary is given by
\begin{equation*}
\partial \mathcal{C}(\va_1, \dots, \va_k) = \bigcup_{i=1}^k
(-1)^{i-1} \mathcal{C}(\va_1, \dots, \va_{i-1}, \va_{i+1}, \dots, \va_k).
\end{equation*}
By the assumption (ii), we have
\begin{equation*}
 \pi_H \Lambda_K(\mathcal{C}(\va_1, \dots, \va_{i-1}, \va_{i+1}, \dots, \va_k))
=\Lambda_K(\mathcal{C}(\va_1, \dots, \va_{i-1}, \va_{i+1}, \dots, \va_k)),
\end{equation*}
so that
\begin{equation*}
\partial \pi_H \Lambda_K(\mathcal{C}(\va_1, \dots, \va_k))
=\bigcup_{i=1}^k
(-1)^{i-1} \Lambda_K(\mathcal{C}(\va_1, \dots, \va_{i-1}, \va_{i+1}, \dots, \va_k))
\end{equation*}
holds.

Since each smooth oriented $(k-2)$-dimensional manifold $\mathcal{C}(\va_1, \dots, \va_{i-1}, \va_{i+1}, \dots, \va_k) \subset H \simeq \R^k$ can be parametrized as in Lemma \ref{lem:2}, we can define the vector 
\begin{equation*}
\overline{\mathcal{C}(\va_1, \dots, \va_{i-1}, \va_{i+1}, \dots, \va_k)} \in H.
\end{equation*}
We apply Proposition \ref{prop:11} to $B:=\mathcal{C}(\va_1, \dots, \va_k)$, so that
\begin{equation*}
\frac{1}{k} \left(
\sum_{i=1}^{k} (-1)^{i-1} \overline{\mathcal{C}(\va_1, \dots, \va_{i-1}, \va_{i+1}, \dots, \va_k)}
\right) \cdot \vx \leq 
|o*\mathcal{C}(\va_1, \dots, \va_k)|_{k}
\end{equation*}
holds for any $\vx \in L$.
It then follows that
\begin{equation*}
\frac{\sum_{i=1}^{k} (-1)^{i-1} \overline{\mathcal{C}(\va_1, \dots, \va_{i-1}, \va_{i+1}, \dots, \va_k)}}
{k |o*\mathcal{C}(\va_1, \dots, \va_k)|_{k}} \in L^\circ.
\end{equation*}
By a similar argument for $B:=\Lambda_K(\mathcal{C}(\va_1, \dots, \va_k))$, we have
\begin{equation*}
\frac{\sum_{i=1}^{k} (-1)^{i-1} \overline{\Lambda_K(\mathcal{C}(\va_1, \dots, \va_{i-1}, \va_{i+1}, \dots, \va_k))}}
{k |o*\pi_H \Lambda_K(\mathcal{C}(\va_1, \dots, \va_k))|_{k}} \in L.
\end{equation*}
The inner product of the above two test points yield
\begin{equation}
\label{eq:17}
\begin{aligned}
& k^2 |o*\mathcal{C}(\va_1, \dots, \va_k)|_{k} |o*\pi_H \Lambda_K(\mathcal{C}(\va_1, \dots, \va_k))|_{k} \\
&\geq 
\left(
\sum_{i=1}^{k} (-1)^{i-1} \overline{\mathcal{C}(\va_1, \dots, \va_{i-1}, \va_{i+1}, \dots, \va_k)}
\right)
\cdot\left(
\sum_{i=1}^{k} (-1)^{i-1} \overline{\Lambda_K(\mathcal{C}(\va_1, \dots, \va_{i-1}, \va_{i+1}, \dots, \va_k))}
\right).
\end{aligned}
\end{equation}

Next we compute the right-hand side.
Since
\begin{equation*}
\mathcal{C}(\va_1, \dots, \va_{i-1}, \va_{i+1}, \dots, \va_k) \subset \Span \{\va_1, \dots, \va_{i-1}, \va_{i+1}, \dots, \va_k \},
\end{equation*}
by Lemma 2.2 (i), $\overline{\mathcal{C}(\va_1, \dots, \va_{i-1}, \va_{i+1}, \dots, \va_k)} \in H$ is perpendicular to every $\va_j$ $(1 \leq j \leq k, j \not=i)$.
Therefore, we can represent it as
\begin{equation*}
 \overline{\mathcal{C}(\va_1, \dots, \va_{i-1}, \va_{i+1}, \dots, \va_k)}
= \beta_i \left(
(1+(k-1)\alpha) \va_i - \alpha \sum_{j=1}^{k} \va_j
\right)
\end{equation*}
for $\beta_i \in \R$.
Taking the orientation of $\mathcal{C}(\va_1, \dots, \va_{i-1}, \va_{i+1}, \dots, \va_k)$ into consideration, the assumption (i) implies $\beta_i=(-1)^{i-1}\beta$.
Consequently, we obtain
\begin{equation*}
\begin{aligned}
\sum_{i=1}^{k} (-1)^{i-1} \overline{\mathcal{C}(\va_1, \dots, \va_{i-1}, \va_{i+1}, \dots, \va_k)}
&=\sum_{i=1}^{k} \beta 
\left(
(1+(k-1)\alpha) \va_i - \alpha \sum_{j=1}^{k} \va_j
\right) \\
&=\beta(1-\alpha) \sum_{i=1}^{k} \va_i.
\end{aligned}
\end{equation*}
By using the assumptions (ii), (iii), the same argument for the corresponding polar part implies
\begin{equation*}
 \overline{\Lambda_K(\mathcal{C}(\va_1, \dots, \va_{i-1}, \va_{i+1}, \dots, \va_k))}
=(-1)^{i-1} \beta^\circ
\left(
(1+(k-1)\alpha) \va_i - \alpha \sum_{j=1}^{k} \va_j
\right),
\end{equation*}
so that
\begin{equation*}
\sum_{i=1}^{k} (-1)^{i-1} \overline{\Lambda_K(\mathcal{C}(\va_1, \dots, \va_{i-1}, \va_{i+1}, \dots, \va_k))}
=\beta^\circ(1-\alpha) \sum_{i=1}^{k} \va_i
\end{equation*}
holds.
Summarizing, the right-hand side of \eqref{eq:17} equals
\begin{equation*}
\beta \beta^\circ (1-\alpha)^2 k (1+(k-1)\alpha).
\end{equation*}

On the other hand, by the assumption (ii), we obtain from Lemma \ref{lem:7} (iii)
\begin{equation*}
\begin{aligned}
&|o*\mathcal{C}(\va_1, \dots, \va_{k-1})|_{k-1}
|o*\Lambda_K(\mathcal{C}(\va_1, \dots, \va_{k-1}))|_{k-1} \\
&=
\overline{\mathcal{C}(\va_1, \dots, \va_{k-1})} \cdot \overline{\Lambda_K(\mathcal{C}(\va_1, \dots, \va_{k-1}))} \\
&= \beta \beta^\circ
\left(
(1+(k-1)\alpha) \va_k - \alpha \sum_{j=1}^{k} \va_j
\right)\cdot
\left(
(1+(k-1)\alpha) \va_k - \alpha \sum_{j=1}^{k} \va_j
\right) \\
&= \beta \beta^\circ (1-\alpha) (1+(k-1)\alpha) (1+(k-2)\alpha),
\end{aligned}
\end{equation*}
which completes the proof.
\end{proof}

\subsection{Schneider's approximation}

Finally, we review the following approximation result \cite{Sc2}*{pp.\,438}.
The map $\Lambda_K$ cannot be defined for every $G$-invariant convex body $K$.
However, to estimate $\mathcal{P}(K)$ from below, it suffices to consider the class $\checkK^n(G)$ only, because $\mathcal{P}(K)$ is continuous with respect to the Hausdorff distance on $\K^n$.

\begin{proposition}[Schneider]
\label{prop:sch}
Let $G$ be a discrete subgroup of $O(n)$.
Let $K \in \K^n(G)$ be a $G$-invariant convex body.
Then, for any $\varepsilon > 0$ there exists a $G$-invariant convex body $K_\epsilon \in \checkK^n(G)$ having the property that $\delta(K,K_\epsilon) < \varepsilon$, where $\delta$ denotes the Hausdorff distance on $\K^n$.
\end{proposition}

\section{The case $G=SO(\triangle^n)$}

\begin{proof}[Proof of Theorem \ref{thm:2}]
The statement is trivial for $n=1$ and is a direct consequence of Mahler's theorem for $n=2$.
So, from now on, we assume that $n \geq 3$.

Denote by $\vv_i$ $(i=1,\dots,n+1)$ the vertices of the simplex $\triangle^n$.
Then we have
\begin{equation*}
 \vv_i \cdot \vv_j = 
\begin{cases}
 1 & \text{ if } i=j, \\
-1/n & \text{ if } i \not= j,
\end{cases}
\end{equation*}
and $O(\triangle^n)$ is nothing but the set of all orthogonal matrices given by the permutations of these vertices.
Moreover, $SO(\triangle^n)=\{ A \in O(\triangle^n); \det A=1 \}$.
For each $\sigma \in \mathfrak{S}_{n+1}$, we denote by $R_\sigma$ the element in $O(\triangle^n)$ which maps $\vv_i$ to $\vv_{\sigma(i)}$ $(i=1,\ldots,n+1)$.
We denote by $(i,j) \in \mathfrak{S}_{n+1}$ the transposition exchanging only $i$ and $j$.
Then $R_{(i,j)}$ represents the reflection of $\R^n$ with respect to the hyperplane with $\vv_j-\vv_i$ as its normal vector, passing through the origin $o$.
And $(a_1,a_2,\ldots,a_k):=(a_1,a_2)(a_2,a_3)\cdots(a_{k-1},a_k) \in \mathfrak{S}_{n+1}$ denotes the cyclic permutation of mutually distinct elements $a_1,\ldots,a_k$ of $\{ 1,2,\ldots,n+1 \}$.

From now on, we consider the group $G:=SO(\triangle^n)$.
By Schneider's approximation, it suffices to consider the case where $K \in \check K^n(G)$.
By a dilation of $K$, we may assume that $\vv_i \in \partial K$ $(i=1,\dots,n+1)$.
We put
\begin{equation*}
B:=\partial K \cap \pos \{\vv_1, \dots, \vv_n\}=\mathcal{C}(\vv_1,\dots,\vv_n), \quad
\tilde{K}:= o*B, \quad
\tilde{K}^\circ := o*\Lambda_K(B).
\end{equation*}
Then, by the symmetry of $K$, we have
\begin{equation*}
|K| \, |K^\circ| = (n+1)^2 |\tilde{K}| \, |\tilde{K}^\circ|.
\end{equation*}
Since
\begin{equation*}
\partial B= \bigcup_{i=1}^n (-1)^{i-1} \mathcal{C}(\vv_1, \dots, \vv_{i-1}, \vv_{i+1},\dots,\vv_n)
\end{equation*}
is a piecewise smooth $(n-2)$-dimensional oriented manifold, by Proposition \ref{prop:11}, we obtain
\begin{equation*}
\left(
\sum_{i=1}^n (-1)^{i-1} \overline{\mathcal{C}(\vv_1, \dots, \vv_{i-1}, \vv_{i+1},\dots,\vv_n)}
\right)
\cdot \vx
\leq
n |o*B| = n |\tilde{K}|
\end{equation*}
for any $\vx \in K$.
It follows that
\begin{equation*}
\frac{1}{n |\tilde{K}|} 
\left(
\sum_{i=1}^n (-1)^{i-1} \overline{\mathcal{C}(\vv_1, \dots, \vv_{i-1}, \vv_{i+1},\dots,\vv_n)}
\right)
\in K^\circ.
\end{equation*}
By the same arguments for $K^\circ$, we have
\begin{equation*}
\frac{1}{n |\tilde{K}^\circ|} 
\left(
\sum_{i=1}^n (-1)^{i-1} \overline{\Lambda_K(\mathcal{C}(\vv_1, \dots, \vv_{i-1}, \vv_{i+1},\dots,\vv_n))}
\right) \in K.
\end{equation*}
Taking the inner product of the above two vectors, we obtain
\begin{equation*}
\begin{aligned}
&n^2 |\tilde{K}| \, |\tilde{K}^\circ| \\
& \geq
\left(
\sum_{i=1}^n (-1)^{i-1} \overline{\mathcal{C}(\vv_1, \dots, \vv_{i-1}, \vv_{i+1},\dots,\vv_n)}
\right)
\cdot
\left(
\sum_{i=1}^n (-1)^{i-1} \overline{\Lambda_K(\mathcal{C}(\vv_1, \dots, \vv_{i-1}, \vv_{i+1},\dots,\vv_n))}
\right).
\end{aligned}
\end{equation*}

Let us first consider the case where $n$ is odd.
Then $R:=R_{(1,\dots,n)} \in G$ and
\begin{equation}
\label{eq:2}
\begin{aligned}
R^{i-1}
\overline{\mathcal{C}(\vv_2, \dots, \vv_n)}
&=
\overline{\mathcal{C}(\vv_{i+1}, \dots, \vv_{n}, \vv_1,\dots,\vv_{i-1})} \\
&=
(-1)^{(i-1)(n-i)}\overline{\mathcal{C}(\vv_1, \dots, \vv_{i-1}, \vv_{i+1},\dots,\vv_n)} \\
&=
(-1)^{i-1}\overline{\mathcal{C}(\vv_1, \dots, \vv_{i-1}, \vv_{i+1},\dots,\vv_n)}
\end{aligned}
\end{equation}
holds by Lemma \ref{lem:2} (ii).
By the symmetry of $K$ and Lemma \ref{lem:1}, we obtain the same representation as \eqref{eq:2} for $\overline{\Lambda_K(\mathcal{C}(\vv_1, \dots, \vv_{i-1}, \vv_{i+1},\dots,\vv_n))}$, and hence
\begin{equation*}
n^2 |\tilde{K}| \, |\tilde{K}^\circ| 
\geq
\left(
\sum_{i=1}^n R^{i-1} \overline{\mathcal{C}(\vv_2, \dots, \vv_n)}
\right)
\cdot
\left(
\sum_{i=1}^n R^{i-1} \overline{\Lambda_K(\mathcal{C}(\vv_2, \dots, \vv_n))}
\right)
\end{equation*}
holds.
Since $\mathcal{C}(\vv_2, \dots, \vv_n) \subset \Span \{\vv_2, \dots, \vv_n \}$, by Lemma \ref{lem:7} (i),
we can put
\begin{equation}
\label{eq:18}
\begin{aligned}
\overline{\mathcal{C}(\vv_2, \dots, \vv_n)} &= a_1(\vv_1-\vv_{n+1}), \\
\overline{\Lambda_K(\mathcal{C}(\vv_2, \dots, \vv_n))} 
&= 
a_1^\circ (\vv_1-\vv_{n+1}) + a_2^\circ \vv_2 + \dots + a_n^\circ \vv_n
\end{aligned}
\end{equation}
for $a_1,a_1^\circ,\ldots,a_n^\circ \in \R$.
Since
$R_{(1,n+1)}R_{(2,\dots,n)} \in G$, we have
\begin{equation*}
\begin{aligned}
R_{(1,n+1)}R_{(2,\dots,n)} \overline{\Lambda_K(\mathcal{C}(\vv_2, \dots, \vv_n))}
&=
\overline{\Lambda_K(R_{(1,n+1)}R_{(2,\dots,n)}\mathcal{C}(\vv_2, \dots, \vv_n))} \\
&=
\overline{\Lambda_K(\mathcal{C}(\vv_3, \dots, \vv_n, \vv_2))}
=
-\overline{\Lambda_K(\mathcal{C}(\vv_2, \dots, \vv_n))}.
\end{aligned}
\end{equation*}
It follows from \eqref{eq:18} that
\begin{equation*}
-a_1^\circ (\vv_1-\vv_{n+1}) + 
a_2^\circ \vv_3 + \dots + a_{n-1}^\circ \vv_n + a_{n}^\circ \vv_2 
=
-a_1^\circ (\vv_1-\vv_{n+1}) 
-a_2^\circ \vv_2 - \dots - a_n^\circ \vv_n,
\end{equation*}
and since $\vv_2, \dots,\vv_n$ are linearly independent,
\begin{equation*}
a_2^\circ=-a_3^\circ=a_4^\circ=\dots=a_{n-1}^\circ=-a_n^\circ.
\end{equation*}
holds.
Combining these equations with $\vv_1+\dots+\vv_n=-\vv_{n+1}$ and \eqref{eq:18}, we obtain
\begin{equation*}
\sum_{i=1}^n R^{i-1} \overline{\Lambda_K(\mathcal{C}(\vv_2, \dots, \vv_n))}
=
\left(a_1^\circ+\dots+a_n^\circ\right)
\left(\vv_1+\dots+\vv_n\right) - n a_1^\circ \vv_{n+1}
=- (n+1) a_1^\circ \vv_{n+1}.
\end{equation*}
On the other hand, by a direct calculation, we have
\begin{equation*}
\sum_{i=1}^n R^{i-1} \overline{\mathcal{C}(\vv_2, \dots, \vv_n)}
=- (n+1) a_1 \vv_{n+1}.
\end{equation*}
Therefore, we obtain
\begin{equation*}
n^2 |\tilde{K}| \, |\tilde{K}^\circ| 
\geq (n+1)^2 a_1 a_1^\circ
= \frac{n(n+1)}{2}
\overline{\mathcal{C}(\vv_2, \dots, \vv_n)} \cdot \overline{\Lambda_K(\mathcal{C}(\vv_2, \dots, \vv_n))},
\end{equation*}
where the second equality holds from \eqref{eq:18} by a direct calculation, because $(\vv_{n+1}-\vv_1) \cdot \vv_i=0$ $(i=2,\dots,n)$.
Consequently, by Lemma \ref{lem:7} (iii)
\begin{equation}
\label{eq:20}
n^2 |\tilde{K}| \, |\tilde{K}^\circ| 
\geq \frac{n(n+1)}{2}
\left|O*\mathcal{C}(\vv_2, \dots, \vv_n)\right|_{n-1}
\left|O*\pi_H \Lambda_K(\mathcal{C}(\vv_2, \dots, \vv_n))\right|_{n-1}
\end{equation}
holds, where $H=\Span\{\vv_2,\dots,\vv_n\}$.

Next, we consider the case where $n$ is even.
In this case, note that $R:=R_{(1,\dots,n)} \not\in G$.
We need to slightly modify the calculation in the former case.
\begin{claim}
$R_{(1,3)}^{i-1} R^{i-1} \overline{\mathcal{C}(\vv_2, \dots, \vv_n)}
=
(-1)^{i-1}\overline{\mathcal{C}(\vv_1, \dots, \vv_{i-1}, \vv_{i+1},\dots,\vv_n)}$
\ \mbox{for} \ $i=1,\ldots,n$.
\end{claim}

Indeed, if $i-1$ is even, then we have $R_{(1,3)}^{i-1} R^{i-1} = R^{i-1} \in G$.
Hence, the claim holds in the same way as \eqref{eq:2}.
If $i-1$ is odd, then $R_{(1,3)}^{i-1} R^{i-1} = R_{(1,3)} R^{i-1} \in G$.
Since $(i-1)(n-i)$ is even and $\vv_1, \vv_3 \in \{\vv_1, \dots, \vv_{i-1}, \vv_{i+1},\dots,\vv_n\}$, we obtain
\begin{equation*}
\begin{aligned}
R_{(1,3)}^{i-1} R^{i-1} 
\overline{\mathcal{C}(\vv_2, \dots, \vv_n)}
&=
\overline{R_{(1,3)} \mathcal{C}(\vv_{i+1}, \dots, \vv_{n}, \vv_1,\dots,\vv_{i-1})} \\
&=
-(-1)^{(i-1)(n-i)} \overline{\mathcal{C}(\vv_1, \dots, \vv_{i-1}, \vv_{i+1},\dots,\vv_n)} \\
&=
-\overline{\mathcal{C}(\vv_1, \dots, \vv_{i-1}, \vv_{i+1},\dots,\vv_n)},
\end{aligned}
\end{equation*}
which completes the proof of the claim.

Now, by the symmetry of $K$, we get the same formula for 
$\overline{\Lambda_K(\mathcal{C}(\vv_2, \dots, \vv_n))}$
as the above claim, and hence
\begin{equation*}
n^2 |\tilde{K}| \, |\tilde{K}^\circ| 
\geq
\left(
\sum_{i=1}^n R_{(1,3)}^{i-1} R^{i-1} \overline{\mathcal{C}(\vv_2, \dots, \vv_n)}
\right)
\cdot
\left(
\sum_{i=1}^n R_{(1,3)}^{i-1} R^{i-1} \overline{\Lambda_K(\mathcal{C}(\vv_2, \dots, \vv_n))}
\right).
\end{equation*}
Since $R_{(2,\dots,n)} \in G$, we have
\begin{equation*}
R_{(2,\dots,n)}\overline{\Lambda_K(\mathcal{C}(\vv_2, \dots, \vv_n))}
=
\overline{\Lambda_K(\mathcal{C}(\vv_3, \dots, \vv_n,\vv_2))}
=
\overline{\Lambda_K(\mathcal{C}(\vv_2, \dots, \vv_n))}.
\end{equation*}
Here, we shall represent $\overline{\mathcal{C}(\vv_2, \dots, \vv_n)}$ and $\overline{\Lambda_K(\mathcal{C}(\vv_2, \dots, \vv_n))}$ such as \eqref{eq:18}.
It follows that
\begin{equation*}
a_1^\circ (\vv_1-\vv_{n+1}) + 
a_2^\circ \vv_3 + \dots + a_{n-1}^\circ \vv_n + a_{n}^\circ \vv_2 
=
a_1^\circ (\vv_1-\vv_{n+1}) + a_2^\circ \vv_2 + \dots + a_n^\circ \vv_n,
\end{equation*}
which yields
\begin{equation*}
 a_2^\circ= \dots =a_n^\circ.
\end{equation*}
Moreover, since $R_{(1,n+1)}R_{(2,3)} \in G$ and $n \geq 4$,
\begin{equation*}
R_{(1,n+1)}R_{(2,3)} \overline{\Lambda_K(\mathcal{C}(\vv_2, \dots, \vv_n))}
=
\overline{R_{(1,n+1)}R_{(2,3)} \Lambda_K(\mathcal{C}(\vv_2, \dots, \vv_n))}
=
-\overline{\Lambda_K(\mathcal{C}(\vv_2, \dots, \vv_n))}
\end{equation*}
holds.
Hence, we have
\begin{equation*}
-a_1^\circ (\vv_1-\vv_{n+1}) + a_2^\circ \vv_3 + a_3^\circ \vv_2 + a_4^\circ \vv_4 + \dots + a_n^\circ \vv_n
=
-a_1^\circ (\vv_1-\vv_{n+1}) - a_2^\circ \vv_2 - \dots - a_n^\circ \vv_n,
\end{equation*}
which yields $a_3^\circ=-a_2^\circ$, and hence
\begin{equation*}
 a_2^\circ= \dots =a_n^\circ=0.
\end{equation*}
It follows that
\begin{equation*}
\begin{aligned}
\sum_{i=1}^n R_{(1,3)}^{i-1} R^{i-1} \overline{\Lambda_K(\mathcal{C}(\vv_2, \dots, \vv_n))} 
&= a_1^\circ \left(
-n \vv_{n+1} + \vv_1 + R_{(1,3)} \vv_2 + \vv_3 + \cdots + R_{(1,3)} \vv_n
\right) \\
&= -(n+1) a_1^\circ \vv_{n+1}.
\end{aligned}
\end{equation*}
Hence, inequality \eqref{eq:20} holds by the same way as in the case where $n$ is odd.
Consequently, inequality \eqref{eq:20} holds for any $n \geq 3$.
Here we change the suffixes of vectors in the right-hand side of \eqref{eq:20} for the sake of the following calculation.
Since $R_{(n,n-1,\dots,2,1)} \in G$ if $n$ is odd and $R_{(n+1,n,\dots,2,1)} \in G$ if $n$ is even, inequality \eqref{eq:20} is equivalent to
\begin{equation}
\label{eq:25}
n^2 |\tilde{K}| \, |\tilde{K}^\circ| 
\geq \frac{n(n+1)}{2}
\left|O*\mathcal{C}(\vv_1, \dots, \vv_{n-1})\right|_{n-1}
\left|O*\pi_H \Lambda_K(\mathcal{C}(\vv_1, \dots, \vv_{n-1}))\right|_{n-1},
\end{equation}
where $H:=\Span\{\vv_1,\dots,\vv_{n-1}\}$.

Now we shall estimate the right-hand side of \eqref{eq:25} by means of Lemma \ref{lem:8} inductively.
Assume that $2 \leq k \leq n-1$.
Putting $\va_i:= \vv_i$ $(i=1,\dots,k)$, we first check the assumption of Lemma \ref{lem:8}.

\begin{claim}
The assumptions (i) and (iii) of Lemma \ref{lem:8} are satisfied.
\end{claim}
Let us consider the transformation $R_\sigma \in O(\triangle^n)$ defined by the permutation $\sigma:=(i,i+1,\ldots,k) \in \mathfrak{S}_{n+1}$.
If $(k-i)$ is even, then $\sigma$ is an even permutation and $R_\sigma \in G$.
Hence,
\begin{equation*}
\begin{aligned}
R_\sigma \mathcal{C}(\vv_1,\dots, \vv_{k-1})
&=\mathcal{C}(\vv_1,\dots, \vv_{i-1}, \vv_{i+1}, \dots, \vv_k), \\
R_\sigma \Lambda_K(\mathcal{C}(\vv_1,\dots, \vv_{k-1}))
&=\Lambda_K(\mathcal{C}(\vv_1,\dots, \vv_{i-1}, \vv_{i+1}, \dots, \vv_k)).
\end{aligned}
\end{equation*}
If $(k-i)$ is odd, then $\sigma$ is an odd permutation and $R_{(n, n+1)} R_\sigma \in G$.
Since $k \leq n-1$, we have
\begin{equation*}
\begin{aligned}
R_{(n, n+1)} R_\sigma \mathcal{C}(\vv_1,\dots, \vv_{k-1})
&=\mathcal{C}(\vv_1,\dots, \vv_{i-1}, \vv_{i+1}, \dots, \vv_k), \\
R_{(n, n+1)} R_\sigma \Lambda_K(\mathcal{C}(\vv_1,\dots, \vv_{k-1}))
&=\Lambda_K(\mathcal{C}(\vv_1,\dots, \vv_{i-1}, \vv_{i+1}, \dots, \vv_k)).
\end{aligned}
\end{equation*}
Therefore, in both cases the claim holds.

\begin{claim}
The assumption (ii) of Lemma \ref{lem:8} is satisfied.
In particular, if $k \leq n-1$, then $\pi_{H_{k-1}}\Lambda_K(\mathcal{C}(\vv_1,\dots, \vv_{k-1}))=\Lambda_K(\mathcal{C}(\vv_1,\dots, \vv_{k-1}))$, where $H_{k-1}:=\Span\{\vv_1,\dots,\vv_{k-1}\}$.
\end{claim}

By the argument in the preceding claim, each $\Lambda_K(\mathcal{C}(\vv_1,\dots, \vv_{i-1}, \vv_{i+1}, \dots, \vv_k))$ is $G$-congruent to $\Lambda_K(\mathcal{C}(\vv_1,\dots, \vv_{k-1}))$.
Hence, it suffices to show that $\Lambda_K(\mathcal{C}(\vv_1,\dots, \vv_{k-1})) \subset H_{k-1}$.

Let $\vx \in \mathcal{C}(\vv_1,\dots, \vv_{k-1}) \subset H_{k-1} \cap \partial K$.
Let us first choose the suffix $j$ with $k+1 \leq j \leq n$.
Since $R_{(j-1, j, j+1)} \in G$ and $R_{(j-1, j, j+1)}$ fixes $H_{k-1}$, we have
\begin{equation}
\label{eq:3}
\Lambda_K(\vx)
=\Lambda_K(R_{(j-1, j, j+1)} \vx)
=R_{(j-1, j, j+1)} \Lambda_K(\vx).
\end{equation}
Now we put
\begin{equation*}
 \Lambda_K(\vx) =a_1^\circ \vv_1 + \cdots + a_n^\circ \vv_n
\end{equation*}
for $a_1^\circ,\ldots,a_n^\circ \in \R$.
When $j=n$, from \eqref{eq:3} we have
\begin{equation*}
\begin{aligned}
a_1^\circ \vv_1 + \cdots + a_n^\circ \vv_n
&=
a_1^\circ \vv_1 + \cdots + a_{n-2}^\circ \vv_{n-2} + a_{n-1}^\circ \vv_n + a_n^\circ \vv_{n+1} \\
&=
(a_1^\circ-a_n^\circ)\vv_1 + \cdots + (a_{n-2}^\circ-a_n^\circ) \vv_{n-2} -a_n^\circ \vv_{n-1} + (a_{n-1}^\circ-a_n^\circ) \vv_n,
\end{aligned}
\end{equation*}
which yields that $a_{n-1}^\circ=a_n^\circ=0$.
In the case where $n=3$, since $3 \leq k+1 \leq j \leq n \leq 3$, $k=2$ holds, and hence $\Lambda_K(\vx) =a_1^\circ \vv_1 \in H_1$.
If $n \geq 4$, then by \eqref{eq:3} we have
\begin{equation*}
a_{j-1}^\circ \vv_{j-1} + a_j^\circ \vv_j + a_{j+1}^\circ \vv_{j+1} =
a_{j-1}^\circ \vv_j + a_j^\circ \vv_{j+1} + a_{j+1}^\circ \vv_{j-1},
\end{equation*}
so that $a_{j-1}^\circ =a_j^\circ = a_{j+1}^\circ$ holds for $k+1 \leq j <n$.
Hence,
\begin{equation*}
 a_k^\circ = \dots = a_n^\circ=0
\end{equation*}
and $\Lambda_K(\vx) \in H_{k-1}$, which verifies the assumption (ii).

\smallskip
\smallskip

Finally we apply Lemma \ref{lem:8} to the right-hand side of \eqref{eq:25}.
Taking the above claim into account, for $k=2,\ldots,n-2$, we put
\begin{equation*}
I_k:=|o*\mathcal{C}(\vv_1, \dots, \vv_k)|_k \, |o*\Lambda_K(\mathcal{C}(\vv_1, \dots, \vv_k))|_k
\end{equation*}
and
\begin{equation*}
I_{n-1}:=|o*\mathcal{C}(\vv_1, \dots, \vv_{n-1})|_{n-1} \, |o*\pi_H \Lambda_K(\mathcal{C}(\vv_1, \dots, \vv_{n-1}))|_{n-1}.
\end{equation*}
By Lemma \ref{lem:8}, we obtain
\begin{equation*}
I_k \geq \frac{n+1}{k(n+2-k)} I_{k-1} \quad \text{ for } \quad k=2,\ldots,n-1.
\end{equation*}
Since $I_1=1$, we have
\begin{equation*}
I_{n-1} \geq \frac{n+1}{(n-1)3}I_{n-2} \geq \dots \geq \frac{2n(n+1)^{n-1}}{(n!)^2}.
\end{equation*}
It follows from \eqref{eq:25} that
\begin{equation*}
|\tilde{K}| \, |\tilde{K}^\circ|
\geq \frac{n+1}{2n} \, \frac{2n(n+1)^{n-2}}{(n!)^2}
= \frac{(n+1)^{n-1}}{(n!)^2}.
\end{equation*}
Consequently, $|K| \, |K^\circ| \geq (n+1)^{n+1}/(n!)^2$ holds.
\end{proof}

\section{The case $G=SO(\Diamond^n)$}

\begin{proof}[Proof of Theorem \ref{thm:1}]
In this section we consider the group $G=SO(\Diamond^n)$.
The statement is trivial for $n=1$ and is a consequence of Mahler's theorem for $n=2$.
So, we assume that $n \geq 3$.
Let $K \in \check K^n(G)$ and $\{\ve_1,\dots,\ve_n\}$ be the standard basis on $\R^n$.
By a dilation of K, we may assume that $\ve_i \in \partial K$ $(i=1,\dots,n)$.
In what follows, we put
\begin{equation*}
B:=\partial K \cap \pos \{\ve_1, \dots, \ve_n\}=\mathcal{C}(\ve_1, \dots, \ve_n), \quad
\tilde{K}:=o*B, \quad
\tilde{K}^\circ := o*\Lambda_K(B).
\end{equation*}
Then by the symmetry of $K$,
\begin{equation*}
|K| \, |K^\circ| = 4^n |\tilde{K}| \, |\tilde{K}^\circ|
\end{equation*}
holds.
As in the proof of Theorem \ref{thm:2}, we obtain
\begin{equation*}
\begin{aligned}
&n^2 |\tilde{K}| \, |\tilde{K}^\circ| \\
& \geq
\left(
\sum_{i=1}^n (-1)^{i-1} \overline{\mathcal{C}(\ve_1, \dots, \ve_{i-1}, \ve_{i+1},\dots,\ve_n)}
\right)
\cdot
\left(
\sum_{i=1}^n (-1)^{i-1} \overline{\Lambda_K(\mathcal{C}(\ve_1, \dots, \ve_{i-1}, \ve_{i+1},\dots,\ve_n))}
\right).
\end{aligned}
\end{equation*}
Now, for each $i=2,\dots,n$, we consider the transformation $R_i \in SO(n)$ defined by
\begin{equation*}
 R_i \ve_i=\ve_1, \quad R_i \ve_1 = -\ve_i, \quad R_i \ve_j = \ve_j \, (j\not=1,i).
\end{equation*}
Then
\begin{equation}
\label{eq:22}
 R_i \overline{\mathcal{C}(\ve_2, \dots, \ve_n)}
=
(-1)^{i} \overline{\mathcal{C}(\ve_1, \dots, \ve_{i-1}, \ve_{i+1},\dots,\ve_n)}.
\end{equation}
holds.
Similarly, we have
\begin{equation*}
 R_i \overline{\Lambda_K(\mathcal{C}(\ve_2, \dots, \ve_n))}
=
(-1)^{i} \overline{\Lambda_K(\mathcal{C}(\ve_1, \dots, \ve_{i-1}, \ve_{i+1},\dots,\ve_n))}.
\end{equation*}
Therefore, we obtain
\begin{equation}
\label{eq:23}
n^2 |\tilde{K}| \, |\tilde{K}^\circ| \\
 \geq
\left(
\left(I- \sum_{i=2}^n R_i\right) \overline{\mathcal{C}(\ve_2, \dots, \ve_n)}
\right)
\cdot
\left(
\left(I- \sum_{i=2}^n R_i\right) \overline{\Lambda_K(\mathcal{C}(\ve_2, \dots, \ve_n))}
\right).
\end{equation}

For $2 \leq i < j \leq n$, we use the transformation $R_{ij} \in SO(n)$ defined by
\begin{equation*}
 R_{ij} \ve_1=-\ve_1, \quad R_{ij} \ve_i = \ve_j, \quad R_{ij} \ve_j = \ve_i, \quad
R_{ij} \ve_k = \ve_k \, (k\not=1,i,j).
\end{equation*}
Then
\begin{equation*}
 R_{ij} \overline{\Lambda_K(\mathcal{C}(\ve_2, \dots, \ve_n))}
= - \overline{\Lambda_K(\mathcal{C}(\ve_2, \dots, \ve_n))}
\end{equation*}
holds.
Here we put
\begin{equation*}
\overline{\Lambda_K(\mathcal{C}(\ve_2, \dots, \ve_n))}
=
a_1^\circ \ve_1 + \cdots + a_n^\circ \ve_n.
\end{equation*}
Then we obtain $a_i^\circ=-a_j^\circ$ for any $2 \leq i < j \leq n$.
If $n \geq 4$, then it follows that $a_2^\circ= \cdots = a_n^\circ=0$.
Hence we have
\begin{equation*}
 \overline{\Lambda_K(\mathcal{C}(\ve_2, \dots, \ve_n))} = a_1^\circ \ve_1,
\end{equation*}
which yields
\begin{equation*}
\left(I- \sum_{i=2}^n R_i\right) \overline{\Lambda_K(\mathcal{C}(\ve_2, \dots, \ve_n))}
=a_1^\circ(\ve_1 + \cdots + \ve_n).
\end{equation*}
This equality also holds for $n=3$.
Indeed, in this case $a_2^\circ=-a_3^\circ$ holds, and hence
\begin{equation*}
\left(I-R_2-R_3\right)\overline{\Lambda_K(\mathcal{C}(\ve_2, \ve_3))}
=\left(I-R_2-R_3\right)\left(a_1^\circ \ve_1 + a_2^\circ \ve_2 - a_2^\circ \ve_3\right)
= a_1^\circ (\ve_1 + \ve_2 + \ve_3).
\end{equation*}
On the other hand, by Lemma \ref{lem:7} (i), we can put
\begin{equation*}
\overline{\mathcal{C}(\ve_2, \dots, \ve_n)} = a_1 \ve_1
\end{equation*}
for $a_1 \in \R$, which yields
\begin{equation*}
\left(I- \sum_{i=2}^n R_i\right) \overline{\mathcal{C}(\ve_2, \dots, \ve_n)}
=a_1(\ve_1 + \cdots + \ve_n).
\end{equation*}
Hence the right hand side of \eqref{eq:23} is equal to
\begin{equation*}
\begin{aligned}
n a_1 a_1^\circ
&=
n
\overline{\mathcal{C}(\ve_2, \dots, \ve_n)} \cdot
\overline{\Lambda_K(\mathcal{C}(\ve_2, \dots, \ve_n))} \\
&=
n
\left|
O*\mathcal{C}(\ve_2, \dots, \ve_n)
\right|_{n-1}
\left|
O*\pi_H \Lambda_K(\mathcal{C}(\ve_2, \dots, \ve_n))
\right|_{n-1}.
\end{aligned}
\end{equation*}
Therefore, we obtain
\begin{equation}
\label{eq:24}
|\tilde{K}| \, |\tilde{K}^\circ| \\
 \geq
\frac{1}{n}
\left|
O*\mathcal{C}(\ve_2, \dots, \ve_n)
\right|_{n-1}
\left|
O*\pi_H \Lambda_K(\mathcal{C}(\ve_2, \dots, \ve_n))
\right|_{n-1},
\end{equation}
where $H:=\Span\{\ve_2,\dots,\ve_n\}$.

Next, for $k \leq n-1$ we shall verify the assumptions of Lemma \ref{lem:8}.
As \eqref{eq:22}, for $i=2,\dots,k$ we have
\begin{equation*}
\begin{aligned}
(-1)^{i} R_i \mathcal{C}(\ve_2, \dots, \ve_k)
&=
\mathcal{C}(\ve_1, \dots, \ve_{i-1}, \ve_{i+1},\dots,\ve_k), \\
(-1)^{i} R_i \Lambda_K(\mathcal{C}(\ve_2, \dots, \ve_k))
&=
\Lambda_K(\mathcal{C}(\ve_1, \dots, \ve_{i-1}, \ve_{i+1},\dots,\ve_k)).
\end{aligned}
\end{equation*}
Since $(-1)^{i} R_i \in O(n)$, the assumptions (i) and (iii) of Lemma \ref{lem:8} are satisfied.
We shall check (ii).
By the symmetry of $K$, it suffices to show that $\Lambda(\vx) \subset \Span\{\ve_2, \dots, \ve_k\}$ for $\vx \in \partial K \cap \Span\{\ve_2, \dots, \ve_k\}$.
For $k < i \leq n$, we have
\begin{equation*}
R_i \Lambda_K(\vx) = \Lambda_K(R_i \vx) = \Lambda_K(\vx),
\end{equation*}
so that we put
\begin{equation*}
 \Lambda_K(\vx) = a_1^\circ \ve_1 + \cdots + a_n^\circ \ve_n,
\end{equation*}
then $a_1^\circ=a_i^\circ$, $-a_1^\circ=a_i^\circ$ holds.
Hence, we obtain
\begin{equation*}
 a_1^\circ =0,\ a_{k+1}^\circ = \dots = a_n^\circ =0,
\end{equation*}
which means that $\Lambda_K(\vx) \in \Span\{\ve_2, \dots, \ve_k\}$.

As in the last paragraph of the proof of Theorem \ref{thm:2} in Section 3, by using \eqref{eq:24} and Lemma \ref{lem:8}, we complete the proof.
\end{proof}

\section{From Mahler's conjecture to symplectic measurements}

Recently, Conjecture \ref{conj:3} has been intensively studied for convex Lagrangian products (see, e.g., \cite{B}, \cite{Ru} and \cite{SL}).
In that case, solutions of Mahler's conjecture provide a powerful tool to study it.
In the last section we prove Proposition \ref{prop:v} and give its applications to Conjecture \ref{conj:3}, other than Corollary \ref{cor:4}, by means of several partial known results of Conjecture \ref{conj:1}.
We first recall the following result, which is a generalization of Theorem \ref{thm:a}.

\begin{proposition}[\cite{AKO}, Remark 4.2]
\label{prop:4.3}
Let $K, T \in \K_0^n$.
Then
\begin{equation*}
 c_{\rm HZ}(K \times T) = \overline{c}(K \times T) = 4\,\mathrm{inrad}_{T^\circ}(K),
\end{equation*}
where $\mathrm{inrad}_{T^\circ}(K):=\sup\{ r>0;\,r T^\circ \subset K \}$ and $\overline{c}$ denotes the cylindrical capacity defined by
\begin{equation*}
 \overline{c}(\Sigma)
:=\inf\{ \pi r^2; \mbox{there is a symplectic embedding from} \ \Sigma \ \mathrm{to} \ (Z^{2n}(r),\omega_0) \}.
\end{equation*}
\end{proposition}

It is well-known that the Gromov width $c_G$ and the cylindrical capacity $\overline{c}$ are the smallest and largest symplectic capacities, respectively, i.e., 
$c_G(M,\omega) \leq c(M,\omega) \leq \overline{c}(M,\omega)$.
It follows from Proposition \ref{prop:4.3} that a simple argument yields Proposition \ref{prop:v}.

\begin{proof}[Proof of Proposition \ref{prop:v}]
By definition, $(\mathrm{inrad}_{T^\circ}(K))\,T^\circ \subset K$ holds.
By the bipolar theorem, we have
\begin{equation*}
 (\mathrm{inrad}_{T^\circ}(K))^{-1}\,T \supset K^\circ,
\end{equation*}
which implies that
\begin{equation}
\label{eq:d}
 (\mathrm{inrad}_{T^\circ}(K))^{-n}\,|T| \geq |K^\circ|.
\end{equation}
Assume that Mahler's conjecture is true for the centrally symmetric convex body $K \subset \R^n$.
Then, by \eqref{eq:d} and the fact that $K \times T$ is a Lagrangian product in $(\R^{2n},\omega)$, we have
\begin{equation*}
\begin{aligned}
\frac{4^n}{n!} \leq |K| \, |K^\circ|
&\leq
(\mathrm{inrad}_{T^\circ}(K))^{-n}\,|K| \, |T| \\
&=
(\mathrm{inrad}_{T^\circ}(K))^{-n}\,\mathrm{vol}(K \times T).
\end{aligned}
\end{equation*}
By Proposition \ref{prop:4.3}, we obtain
\begin{equation*}
\frac{\overline{c}(K \times T)^n}{n!}
=
\frac{(4\,\mathrm{inrad}_{T^\circ}(K))^n}{n!}
\leq
\mathrm{vol}(K \times T).
\end{equation*}
Since $c \leq \overline{c}$ for any symplectic capacity $c$, we complete the proof.
\end{proof}

Now we apply Proposition \ref{prop:v} to several cases.
The solution of Mahler's conjecture (Conjecture \ref{conj:1}) for $n=2,3$ immediately implies

\begin{corollary}
\label{cor:5.2}
Let $n=2$ or $3$.
Let $K$ and $L$ be a centrally symmetric convex bodies in ${\mathbb R}^n$.
Then for the Lagrangian product $\Sigma:=K \times L \subset ({\mathbb R}^{2n},\omega_0)$ the inequality $c(\Sigma)^n \leq n!\,\mathrm{vol}(\Sigma)$ holds for any symplectic capacity $c$.
\end{corollary}

One of the famous partial answer of Conjecture \ref{conj:1} is the case of $1$-unconditional bodies \cite{SR}.
The result yields

\begin{corollary}
\label{cor:5.3}
Let $K$ be an $1$-unconditional body in ${\mathbb R}^n$ and $L \in \K_0^n$.
Then for the Lagrangian product $\Sigma:=K \times L \subset ({\mathbb R}^{2n},\omega_0)$ the inequality $c(\Sigma)^n \leq n!\,\mathrm{vol}(\Sigma)$ holds for any symplectic capacity $c$.
\end{corollary}

\begin{remark}
\begin{enumerate}[(i)]
\item Corollary \ref{cor:5.3} is a generalization of \cite{SL}*{Theorem 1.8}.
\item It is interesting to compare this result to [B, Theorem 5.2], which states that $c_{\rm HZ}(P \times L)^n \leq n!\,\mathrm{vol}(P \times L)$ for any convex body $L \subset \R^n$, where $P \in \K_0^n$ denotes a parallelotope and is an example of 1-unconditional bodies.
Note that $L$ is not assumed to be centrally symmetric.
On the other hand, Corollary \ref{cor:5.3} assumes that $L$ is centrally symmetric, but $K \in \K_0^n$ is arbitrary 1-unconditional body.
See also \cite{Ru}*{Theorem 1.1} for $n=2$.
\end{enumerate}
\end{remark}

Any other partial results of Mahler's conjecture (see, e.g., \cite{Sc}*{Section 10.7}) bear partial answers of Viterbo's conjecture.
On the other hand, it is known that Conjecture \ref{conj:3} is true for convex bounded domain $\Sigma \subset (\R^{2n},\omega_0)$ {\it near} the symplectic ball $B^{2n}$ (see \cite{AB}*{Corollary 2} for details).
In this case $\Sigma$ is not necessarily a Lagrangian product.
However, it is difficult to determine whether a given Lagrangian product is near $B^{2n}$.

Finally, we note that there is other approach to Conjecture \ref{conj:3}, so-called {\it strong Viterbo's conjecture} which predicts that all symplectic capacities coincide for each convex bounded domain $\Sigma \subset (\R^{2n},\omega_0)$.
For this direction, we refer \cite{GHR} and references therein.

\end{document}